\documentclass[a4paper,11pt]{article}

\usepackage{a4wide}
\usepackage[utf8]{inputenc}
\usepackage[T1]{fontenc}
\usepackage[english]{babel}
\usepackage{amsmath}
\usepackage{amssymb}
\usepackage{url}
\usepackage{multirow}
\usepackage{amsfonts}
\usepackage{amsthm}
\usepackage{indentfirst}
\usepackage{bbm}
\usepackage{color}
\usepackage{bigints}

\numberwithin{equation}{section}

\usepackage{graphicx}
\usepackage{tikz}
\usepackage{float}

\newtheorem{theorem}{Theorem}[section]

\newtheorem{example}[theorem]{Example}
\newtheorem{remark}[theorem]{Remark}

\newcommand{\ii}{{\mathrm{i}}}

\newcommand{\phib}{\boldsymbol\phib}

\def\C{\mathcal{C}}

\def\Xint#1{\mathchoice
   {\XXint\displaystyle\textstyle{#1}}%
   {\XXint\textstyle\scriptstyle{#1}}%
   {\XXint\scriptstyle\scriptscriptstyle{#1}}%
   {\XXint\scriptscriptstyle\scriptscriptstyle{#1}}%
   \!\int}
\def\XXint#1#2#3{{\setbox0=\hbox{$#1{#2#3}{\int}$}
     \vcenter{\hbox{$#2#3$}}\kern-.5\wd0}}

\def\dashint{\Xint-}

    \graphicspath{{figures/}}
\date{\today}

\begin{document}
\title{Approximation of the Hilbert Transform on the unit circle}

\author{ 
Luisa Fermo\thanks{Department of Mathematics and Computer Science, University of Cagliari,  Via Ospedale 72, 09124, Italy (\tt fermo@unica.it)} \thanks{Corresponding author} \and
Valerio Loi\thanks{Department of Science and High Technology, Insubria University, Via Valleggio 11, Como, 22100, Italy (\tt vloi@studenti.uninsubria.it)}}

\maketitle

\begin{abstract}
The paper deals with the numerical approximation of the Hilbert transform on the unit circle using Szeg\"o and anti-Szeg\"o  quadrature formulas. These schemes exhibit maximum  precision with oppositely signed errors and allow for improved accuracy through their averaged results. Their computation involves a free parameter associated with the corresponding para-orthogonal polynomials. Here, it is suitably chosen to construct a Szeg\"o and anti-Szeg\"o formula whose nodes  are  
strategically distanced from the singularity of the Hilbert kernel.
Numerical experiments demonstrate the accuracy of the proposed method.

\medskip

\end{abstract}

\section{Introduction}
The present paper aims to compute the Hilbert transform which frequently arises in physics and engineering problems such as in signal analysis \cite{Marple1999}, in the Benjamin-Ono
equation \cite{Benjamin1967,Ono1975}, in the mechanical vibration/diagnostics [5,6], and in the solution to matrix-valued Riemann-Hilbert problems \cite{Olver2011}. It has the form
\begin{equation}\label{Cauchy}
(If)(z)= \frac{1}{\pi} \dashint_{\Gamma} \frac{f(t)}{t-z} dt, \qquad z \in \Gamma,
\end{equation}
where the integral is understood in the Cauchy principal value sense, $\Gamma$ is the unit circle in $\mathbb{C}$, i.e.,
$\Gamma=e^{\ii {\mathbb{T}}} = \{ z \in \mathbb{C} : |z|=1\}$, $\ii=\sqrt{-1}$,  $\mathbb{T}=[-\pi,\pi],$
and $f: \Gamma \to \mathbb{C}$ is a known function. 

It is well known \cite[p.104]{Kress99} that, setting $t=e^{\ii \theta}$ and $z=e^{\ii \phi}$ with $\theta,\phi \in \mathbb{T}$, by virtue of
\begin{align*}
\frac{\ii e^{\ii \theta}}{e^{\ii \theta}-e^{\ii \phi}}= \ii \left[
\frac{e^{ (\frac{\theta-\phi}{2}) \ii}}{e^{ (\frac{\theta-\phi}{2}) \ii}-e^{ -(\frac{\theta-\phi}{2}) \ii}} \right]=\frac{1}{2} \cot{\left(\frac{\theta-\phi}{2} \right)+\frac{\ii}{2}},
\end{align*}
one gets
\begin{align}\label{If2}
(If)(e^{\ii \phi})&= \frac{1}{2 \pi} \dashint_{-\pi}^\pi  
\cot{\left(\frac{\theta-\phi}{2} \right)} f(e^{\ii \theta}) d \theta  + \frac{\ii}{2\pi} \int_{-\pi}^\pi  f(e^{\ii \theta}) d \theta, \quad \phi \in \mathbb{T}.
\end{align}

Obviously, from the numerical point of view, the first integral
\begin{equation}\label{Hilbert}
(Hf)(\phi)=\frac{1}{2 \pi} \dashint_{-\pi}^\pi  
\cot{\left(\frac{\theta-\phi}{2} \right)} f(e^{\ii \theta}) d \theta , \quad \phi \in \mathbb{T},
\end{equation}
is the most delicate because of the presence of a singular kernel, known as the Hilbert kernel. Sometimes, one refers to integral \eqref{Hilbert} as circular Hilbert transform and for its numerical approximation practical and accurate numerical methods have been developed; see, for instance,  \cite{chawla1974,jin1988,jinyuan1998,micchelli2013,Olver2011,Sun2019}.

In this paper, we propose to evaluate \eqref{Hilbert} by approximating the integral through the Szeg\"o and anti-Szeg\"o formulas, appropriately constructed to treat the singularity of the Hilbert kernel. 
Szeg\"o rules are commonly applied to integrate periodic functions on the unit circle in the complex plane. A $n$-point Szeg\"o quadrature rule is an interpolatory scheme that integrates exactly Laurent polynomials of degree at most $n-1$, the high degree as possible \cite{jones}. The coefficients of such rules are positive and the 
nodes are simple, lie on the unit circle, and are the zeros of the para-orthogonal polynomials associated with the measure of the integral. Moreover, both can be easily computed: the nodes are the eigenvalues of a unitary upper Hessenberg matrix and the coefficients are the magnitude squared of the first component of the eigenvectors \cite{gragg}.  Thus, Szeg\"o formulae are very similar to the Gauss-Christoffel rules \cite{Gauss,Gautschi}: the role of algebraic polynomials, orthogonal polynomials, and Jacobi matrix is played by Laurent polynomials, para-orthogonal polynomials, and Hessenberg matrix, respectively. As a related problem, it is interesting to observe that the study of the node distribution and localization has also been carried out under various assumptions and using tools from asymptotic linear algebra as in \cite{Golinskii2007,Kuijla2001,Tyrty2003}; for the case of Hessenberg structures the analysis in more involved but still manageable as proved in \cite{bogoya2024,coussement2008}. The main difference between Szeg\"o and Gauss-Christoffel rules is that para-orthogonal polynomials are characterized by a free parameter $\tau$ having a unitary modulus. As shown in \cite{Bulthell2001}, for specific values of $\tau$, Szeg\"o rules coincide with some Gauss-Christoffell formula in $[-1,1]$ and one can fix $\tau$ so that the rule has one or two fixed node in the set $\{\pm 1\}$, as described in \cite{Jagels2007}. 
Anti-Szeg\"o rules are introduced in \cite{KimReichel}. They are similar to the anti-Gauss rule developed by Laurie \cite{Laurie1} and are constructed so that the quadrature error is a negative multiple of the integration error provided by the Szeg\"o rule. An interesting result given in \cite[Theorem 3.4]{KimReichel} design the choice of the parameter $\tau$ to construct the $n$-point anti-Szeg\"o rule. Pairs of Szeg\"o and anti-Szeg\"o rule bracket the integral and this implies that a convex combination of the two rules produces a new rule that is more accurate.  

This paper contains at first new results concerning the anti-Szeg\"o formula. In fact, we show that, similarly to the Szeg\"o rule,  a connection between the anti-Szeg\"o formula  and the anti-Gauss rule based on Chebychev nodes in $[-1,1]$ occurs. Then, we propose new quadrature schemes to approximate the integral \eqref{Hilbert}. They are rules of Szeg\"o and anti-Szeg\"o type having a prescribed node that depends on the singularity $\phi$ of the Hilbert kernel so that the two rules have simultaneously zeros sufficiently far from the singularity. We prove the pointwise stability of the two rules, except for a factor $\log{n}$, and provide an estimate for the quadrature errors. For both results, the study of the punctual behavior of the Hilbert transform for continuous functions has been crucial.  Concerning the quadrature error, we give an estimate in terms of the error of best polynomial approximation of the integrand function, getting an optimal order of convergence in the case when the integrand is a function of the Sobolev space, in accordance with \cite{stolle1992}.

The paper is organized as follows. Section \ref{sec:spaces} outlines the spaces in which we will study the integral \eqref{Cauchy}. Section \ref{sec:schemes} focuses on the quadrature formulae we need for the approximation of the transform \eqref{Cauchy} by proving also new results for the anti-Szeg\"o rule. Section \ref{sec:method} contains the algorithm we propose to approximate the Hilbert transform \eqref{Hilbert} and Section \ref{sec:tests} shows its performance through some numerical examples. Finally, in Section \ref{sec:conclusions} we draw some conclusions and briefly sketch possible research perspectives.

\section{Function spaces}\label{sec:spaces}
Let $C(\Gamma)$ be the set of all continuous functions on the unit circle such that
$$ \|f\|_\infty =\sup_{x \in \Gamma}|f(x)|<\infty, $$
and let us also introduce the space $C^0_{2\pi}(\mathbb{T})$ that is the set of all $2\pi$-periodic and continuous functions in $\mathbb{T}=[-\pi,\pi]$ equipped with the norm
$$\|f\|_\infty=\max_{\theta \in \mathbb{T}}|f(\theta)|.$$
It turns out that $C^0_{2\pi}(\mathbb{T}) \equiv C(\Gamma)$ and $C^0_{2\pi}(\mathbb{T}, \| \cdot\|_\infty) \equiv C(\Gamma, \| \cdot\|_\infty)$ is a Banach space.

A basic tool to study the approximation of functions $f \in C^0_{2 \pi}(\mathbb{T})$ is  the $k$-th modulus of smoothness defined as
\begin{equation*}
\omega^k(f,t)=\sup_{0<h \leq t} \max_{x \in I_{hk}}|\Delta_h^k f(x)|,
\end{equation*}
where $I_{hk}=[-\pi+kh,\pi-kh]$, $k \in \mathbb{N}$. $t \in \mathbb{R}^+$, and $$\Delta_h f(x)=f(x+h)-f(x), \qquad \Delta^k_h f = \Delta_h (\Delta^{k-1}_h f).$$

For smoother functions, we consider the Sobolev spaces of index $r \geq 1$
$$W^r= \{f \in C^0_{2 \pi}(\mathbb{T}): \, f^{(r-1)} \in AC, \quad \|f^{(r)}\|_\infty < \infty \},$$
where $AC$ denotes the set of all functions that are absolutely continuous on every closed subinterval of $(-\pi, \pi)$.
We endow this space by the norm
$$\|f\|_{W^r}=\|f\|_\infty + \|f^{(r)}\|_\infty.$$
Functions belonging to Sobolev spaces are also useful to define the $K$-functional defined as
$$K^k(f,t)=\inf_{g \in W^k} \left \{\|f-g\|_\infty+t^k \|g^{(k)}\|_\infty \right\},$$
which is also fruitful for characterizing the smoothness of a function $f \in C^0_{2 \pi}(\mathbb{T})$, according to the following result
\begin{equation}\label{estimateK}
\mathcal{C}^{-1} \omega^k(f,t) \leq K^k(f,t) \leq \mathcal{C} \omega^k(f,t),
\end{equation}
where $\mathcal{C}$ is a positive constant depending only on $k$.

Let us also introduce the space $L^2(\Gamma)$ as the set of all measurable functions $f$ on the unit circle $\Gamma$ such that
$$\|f\|_{L^2(\Gamma)}= \sqrt{\langle f, f \rangle}< \infty, $$
with the inner product defined as
$$\langle f,g \rangle = \frac{1}{2\pi} \int_{-\pi}^\pi \overline{f(e^{\ii \theta})}  g(e^{\ii \theta}) d \theta,$$
where the bar denotes complex conjugation.

Let us note that if $f: \Gamma \to \mathbb{C}$ is a continuous function, then there exists \cite[Theorem 6]{Bulthell1991} a polynomial $L_n \in \Lambda_{-n,n}$ 
where $\Lambda_{-n,n}$ denotes the linear space of Laurent polynomials of the form 
\begin{equation}\label{Laurent}
L_{n}(z)=\sum_{k=-n}^n c_k \, z^k, \qquad c_k, z \in \mathbb{C},
\end{equation} 
so that
$$|f(z)-L_n(z)|< \epsilon, \qquad \epsilon>0. $$

This is a consequence of the classical Weierstrass Approximation Theorem and suggests us to define the error of best polynomial approximation for $f \in C(\Gamma)$ as
$$E_n(f)= \inf_{L_n \in \Lambda_{-n,n}} \|f-L_n\|_\infty.$$

Let us also recall that Laurent polynomials can be identified with trigonometric polynomials $t_n$, after the change of variable $z=e^{\ii \theta}$ in \eqref{Laurent},

$$t_n(\theta)=\frac{a_0}{2}+\sum_{k=1}^n [a_k \cos{(k \theta)}+ b_k \sin{(k \theta)}],$$
with $a_0=2c_0 \in \mathbb{C}$, $a_k=c_k+c_{-k}\in \mathbb{C}$, and $b_k=\ii(c_k-c_{-k})\in \mathbb{C}$, for each $k \geq 1$. Hence, these polynomials are well suited for the approximation of periodic functions. 

\section{Quadrature schemes}\label{sec:schemes}
In this section, we mention the quadrature rule we adopt for our approximation. In the first paragraph, we recall some basic facts of the well-known Szeg\"o rule whereas in the second one, we describe the anti-Szeg\"o rule by showing connection with the classical anti-Gauss quadrature formula associated to the first-kind Chebychev weights. According to our knowledge, this relation is new. Finally, in the third section, we define the more accurate averaged scheme.   

\subsection{Szeg\"o rules}
Let us consider  
\begin{equation}\label{I}
\mathcal{I}(f)= \frac{1}{2 \pi} \int_{-\pi}^\pi f(e^{i\theta}) d\theta.
\end{equation}
The $n$-point Szeg\"o quadrature rule applied to the above integral yields
\begin{equation}\label{szego}
\mathcal{I}(f)= \frac{1}{n}\sum_{k=1}^n f(z_k) +e_n(f):=S_n(f)+e_n(f),
\end{equation} 
where $z_k \in \Gamma$ and $e_n(f)$ is the remainder term.  

It is well known that the Szeg\"o rule integrates exactly Laurent polynomials of degree at most $n-1$, i.e. 
\begin{equation}\label{exactness}
e_n(f)=0, \qquad \forall f \in \Lambda_{-(n-1),n-1},
\end{equation}
and that the quadrature nodes $z_k$ are the zeros of the following para-orthogonal polynomial \cite{gragg,jones} 
\begin{equation}\label{polinomio}
{\psi}_n(z;\tau) = z^n+\tau, \qquad n=0,1,\dots
\end{equation}
where $\tau \in \Gamma$ is arbitrary but fixed. They are given by 
$$z_k=e^{\ii (\theta+2k\pi)/n}, \qquad \tau=-e^{\ii \theta}, \quad k=1,\dots,n $$
and are the eigenvalues of the following unitary upper Hessenberg matrix 
\begin{equation*} 
{H}_{n}(\tau) =
\begin{bmatrix}
\mathbf{0}^T_{n-1} & -\tau \\
I_{n-1} & \mathbf{0}_{n-1} \\
\end{bmatrix} \in \mathbb{R}^{n \times n},
\end{equation*}
where $I_k$ is the identity matrix of order $k$ and $\mathbf{0}^T_{k}$ is the null vector of order $k$, with $k=1,2,\dots$.  For completeness, we also recall that the weights of the   quadrature scheme \eqref{szego}, i.e., $ \lambda_k:=\frac{1}{n}$ are related to the matrix $H_n$. In fact, they are the magnitude squared of the first component of its eigenvectors, normalized to have unit length.

If in \eqref{polinomio} we fix $\tau=1$, Szeg\"o quadrature formula \eqref{szego} is strictly related  to the classical Gauss-Chebychev  rule 
$$\mathcal{G}_n(F)= \frac{1}{2 n} \sum_{j=1}^n  F(x_j), \quad x_j= \cos \left( \frac{2j-1}{2n}\pi\right), \quad  $$
for approximating integrals of the type $$I(F)=\frac{1}{2 \pi}\int_{-1}^1 F(x) \frac{dx}{\sqrt{1-x^2}}, $$
where $f(e^{\ii \theta})=\frac{1}{2} F\left( \frac{e^{\ii \theta}+e^{-\ii \theta}}{2}\right)$. In fact, it is well known that
$\mathcal{G}_n(F)=I(F),$ $\forall F \in \mathbb{P}_{2n-1},$
where $\mathbb{P}_{2n-1}$ denotes the set of all polynomials of degree at most $2n-1$. Then, by using this property, setting $x_j=\cos{\theta_j}$, and defining 
$z_j=e^{\ii \theta_j}$, $z_{n+j}=e^{-\ii \theta_j}$ for $j=1,...,n,$
it is possible to prove that the formula $\mathcal{G}_n(F)$
coincides with the $2n$-point Szeg\"o quadrature formula $S_{2n}(f)$ for the integral \eqref{I}; see  \cite[Theorem 3.2 and Theorem 3.3]{Bulthell2001}.

For our aim, let us also recall that it is possible to construct Szeg\"o rule with a prescribed node $z_p=e^{\ii \theta_p} \in \Gamma$, leading to the so-called Szeg\"o-Radau scheme. As described in \cite{Jagels2007}, this is possible simply by fixing in \eqref{polinomio}
\begin{equation}\label{taup}
\tau:=- z_p \frac{\phi_{n-1}(z_p)}{z_p^{n-1}\overline{\phi_{n-1}(z_p)}}=-e^{\ii n \theta_p}, \qquad \text{with} \qquad \phi_{n-1}(z)=z^{n-1}.
\end{equation}

Let us now introduce some notation which will be useful in the sequel. Given two distinct nodes $u, w \in \Gamma$, we introduce the set 
$$(u,w)=\{ \phi \in \Gamma : \arg(u)<\arg(\phi)<\arg(w) \}.$$

Then, we assume that the set of quadrature nodes $\{ z_1,z_2,\dots z_n\} \in \Gamma$ is cyclicly ordered that is the set of points $(z_i, z_{i+1})_{i=1}^{n-1}$ and $(z_n,z_1)$ does not contain other  $z_i$. 

Moreover, we will say that the set $\{ z_1,z_2,\dots z_n\} \in \Gamma$ strictly interlace the set of points $\{\tilde{z}_1,\tilde{z}_2,\dots \tilde{z}_n\} \in \Gamma$ if after a cyclic permutation of $\tilde{z}_i$
we have $\tilde{z}_i \in (z_i,z_{i+1})$ with $i=1,2,\dots,n-1$, and $\tilde{z}_n \in (z_n,z_{1})$. Obviously, one also has that ${z}_i \in (\tilde{z}_i,\tilde{z}_{i+1})$ with $i=1,2,\dots,n-1$, and ${z}_n \in (\tilde{z}_n,\tilde{z}_{1})$.

The next theorem states an important property of the polynomial \eqref{polinomio} which in general holds true for all para-orthogonal polynomials
\cite{Simon2007}.
\begin{theorem}\label{interlacing}
The $n$ zeros of the polynomial $\hat{\psi}_n(z;\tau)$ strictly interlace the zeros of the polynomial $\hat{\psi}_n(z;-\tau)$.
\end{theorem}

About the error of Szeg\"o schemes \eqref{szego}, error bounds were given in \cite{Bulthell1991} for analytic functions. 
The next theorem furnishes an estimate for the error $e_n(f)$ of the Szeg\"o scheme \eqref{szego}, in terms of the error of best polynomial approximation.
\begin{theorem}\label{stimaerrore}
Let $f \in C^0_{2 \pi}(\mathbb{T})$. Then
$$|e_n(f)| \leq 2 E_n(f).$$
\end{theorem}
\begin{proof}
Let $L^*_n$ be the best Laurent polynomial approximation for $f$. By virtue of \eqref{exactness} we can write
\begin{align*}
|e_n(f)|=|e_n(f-L^*_n)| & \leq  \frac{1}{2 \pi}\int_{-\pi}^\pi | f(e^{\ii \theta})-L^*_n(e^{\ii \theta})| d \theta + \frac{1}{n}\sum_{k=1}^n |f(z_k)-L^*_n(z_k)| \\ & \leq 2 \|f-L^*_n\|_\infty,
\end{align*}
i.e. the statement.
\end{proof} 

\subsection{Anti-Szeg\"o quadrature rules} 
Let us consider again integral \eqref{I}. 
The anti-Szeg\"o quadrature rule \cite{jones,KimReichel} is a formula of the form  
\begin{equation}\label{anti-szego}
\mathcal{I}(f)=\frac{1}{n}\sum_{k=1}^{n}  f(\tilde{z}_k) +\tilde{e}_{n}(f):=\tilde{S}_{n}(f)+\tilde{e}_{n}(f),
\end{equation}
where the nodes 
\begin{equation}\label{zerosanti}
\tilde{z}_k =e^{\ii (\theta+(2k+1)\pi)/n} \in \Gamma, \qquad -\tau=e^{\ii \theta}
\end{equation}
are the zeros of the  polynomial ${\psi}_n(z;-\tau)=z^n-\tau$ and the remainder term is such that
\begin{equation}\label{errorsanti}
\tilde{e}_{n}(p)=-e_n(p), \qquad \forall p \in \Lambda_{-n,n}.
\end{equation}

Then, the anti-Szeg\"o formula is exact for Laurent polynomials of degree at most $n-1$ i.e.
$$\tilde{S}_{n}(p)=\mathcal{I}(p), \qquad \forall p \in \Lambda_{-(n-1),n-1} $$
and identity \eqref{errorsanti} implies 
\begin{equation}\label{Ip}
\mathcal{I}(p)=\frac{1}{2}[\tilde{S}_{n}(p)+{S}_{n}(p)], \qquad \forall p \in \Lambda_{-n,n}
\end{equation}
and 
$$\tilde{S}_{n}(p) \leq \mathcal{I}(p) \leq {S}_{n}(p), \qquad \textrm{or} \qquad {S}_{n}(p) \leq \mathcal{I}(p) \leq \tilde{S}_{n}(p).$$

Note that, according to Theorem \ref{interlacing}, the zeros of the anti-Szeg\"o rule interlace the nodes of the Szeg\"o formula and
$$|\arg(z_k)-\arg(\tilde{z}_k)|=\frac{\pi}{n}.$$
The next two theorems state a connection between the anti-Szeg\"o formula \eqref{anti-szego} and the anti-Gauss rule approximating integrals of the type
\begin{equation}\label{Ig}
I(g)=\frac{1}{2 \pi}\int_{-1}^1 g(x) \frac{dx}{\sqrt{1-x^2}}. 
\end{equation}
Anti-Gauss rules have been introduced for the first time in \cite{Laurie1} and several extensions and generalizations have been introduced recently; see, for instance,  \cite{Notaris2022,PranicReichel,ReichelSpalevic2021}. Referring to the integral \eqref{Ig}, it reads as 
\begin{equation}\label{anti-Gauss}
\tilde{\mathcal{G}}_{n+1}(g)=\frac{1}{2 \pi}\sum_{j=1}^{n+1} \tilde{\lambda}_j g(\tilde{x}_j),  
\end{equation}
where (see, for instance \cite{Laurie1} and \cite[Theorem 2]{DFR2020})
\begin{equation}\label{nodesweights} 
\tilde{x}_j=\cos{\theta_j}:=\cos{\left(\frac{(n-j+1)\pi}{n}\right)}, \qquad \tilde{\lambda}_k=\begin{cases} \frac{\pi}{n}, & j=1,n+1, \\
\frac{\pi}{2n}, & j=2,\dots,n,
 \end{cases}
\end{equation}
and is such that
\begin{equation}\label{prop_anti}
\tilde{\mathcal{G}}_{n+1}(g)=2I(g)-\mathcal{G}_n(g), \qquad \forall g \in \mathbb{P}_{2n+1}.
\end{equation}

\begin{theorem}
Let us consider the (n+1)-point anti-Gauss formula \eqref{anti-Gauss}-\eqref{nodesweights} for the integral \eqref{Ig}. Then, the formula
\begin{equation}\label{Q2n+1}
\mathcal{Q}_{2(n+1)}(f)=\frac{1}{2 \pi}\sum_{j=1}^{2(n+1)} w_j f(z_j)= \frac{1}{2 \pi}\sum_{j=1}^{n+1} w_j[f(z_j)+f(\bar{z_j})],
\end{equation}
 with
 \begin{equation}\label{newzeros}
 z_j:=
 \begin{cases}
 e^{\ii \theta_j}, & j=1,\dots,n+1, \\
  e^{-\ii \theta_j}, & j=n+2,\dots,2(n+1),
 \end{cases}, \qquad 
 w_j=w_{n+j}=\tilde{\lambda}_j, \quad j=1,\dots,n+1 
 \end{equation}
coincides with the $2(n+1)$-point anti-Szeg\"o rule having the nodes \eqref{zerosanti} with $\tau=1$. 
\end{theorem}
\begin{proof}
The proof consists of showing that for each Laurent polynomial $p \in \Lambda_{-(2n+1),2n+1}$, the formula $\mathcal{Q}_{2(n+1)}$ satisfies \eqref{Ip}, that is
$$\mathcal{Q}_{2(n+1)}(p)= 2\mathcal{I}(p)- {S}_{2(n+1)}(p), \quad \forall p \in \Lambda_{-(2n+1),2n+1}.$$ 
To this end, it is sufficient to prove the above relation for $p(z)=z^k$ with $0 \leq k \leq 2n+1$. The same proof works for $-(2n+1) \leq k<0$. 

By \eqref{Q2n+1} and using \eqref{newzeros} we can write 
\begin{align*}
\mathcal{Q}_{2(n+1)}(z^k)&= \frac{w_1}{2\pi}(z_1^k+\bar{z}_1^k)+\frac{1}{2\pi} \sum_{j=2}^{n} w_j (z_j^k+\bar{z}_j^k)+\frac{1}{2\pi} (z_{n+1}^k+\bar{z}_{n+1}^k) \\ & = \frac{1}{n} \cos{(\pi)}+ \frac{1}{2n} \sum_{j=2}^{n-1} \cos{(k \theta_j)}+\frac{1}{n} \cos{(0 )} \\ &= 2 \tilde{\mathcal{G}}_{n+1}(T_k) = 2 (2I(T_k)-\mathcal{G}_n(T_k)), 
\end{align*}
where $T_k(x)=\cos{(k \arccos(x))}$, $0 \leq k \leq 2n+1$ is the polynomial of Chebychev of the first kind and the last identity follows by \eqref{prop_anti}. At this point the proof is completed by observing that 
$$2 I(T_k)=\frac{1}{2 \pi}\int_{-\pi}^\pi \cos{(k \theta)} d \theta = \frac{1}{2 \pi} \int_{-\pi}^\pi [\cos{(k \theta)}+ \ii \sin{(k \theta})] d \theta =  \mathcal{I}(z^k), $$ 
and
$$\mathcal{G}_n(T_k) = \frac{1}{2}{S}_{2n}(\cos{(k \theta)})=\frac{1}{2} {S}_{2n}(z^k),$$
where the Szeg\"o formula is constructed with $\tau=1$.
\end{proof}
\begin{theorem}
Let us consider the $2(n+1)$-point anti-Szeg\"o formula \eqref{anti-szego} whose nodes $\{\tilde{z}_k\}_{k=1}^{2(n+1)}$ are defined in \eqref{zerosanti} with $\tau=1$. Then, setting
$$ 
\tilde{z}_{n+j}= \bar{\tilde{z}}_{j}=e^{-\ii \theta_j}, \qquad  
\tilde{x}_j=\cos{\theta}_j,  \qquad \tilde{\lambda}_j= \tilde{w}_j \qquad j=1,2,\dots,n+1,
$$
the formula 
$$\mathcal{Q}_{n+1}(g)= \sum_{j=1}^{n+1} \tilde{\lambda}_j g(\tilde{x}_j)$$ 
coincides with the $(n+1)$-point anti-Gauss formula for the integral $I(g)$ given in \eqref{Ig}.  
\end{theorem}
\begin{proof}
We need to prove that formula $\mathcal{Q}_{n+1}(g)$ is  such that
$$\mathcal{Q}_{n+1}(x^k)=2I(x^k)-\mathcal{G}_n(x^k), \qquad k=1,\dots ,2n+1.$$
First, let us note that setting
$$x= \left(\frac{z+z^{-1}}{2} \right),$$
then $x^k=L(z) \in \Lambda_{-k,k}$ is a Laurent polynomial having real coefficients. Therefore, we can write
\begin{align*}
\mathcal{Q}_{n+1}(x^k)= \sum_{j=1}^{n+1} \tilde{\lambda}_j \tilde{x}_j^k& = \tilde{\lambda}_1 \tilde{x}_1^k+ \sum_{j=2}^{n}  \tilde{\lambda}_j \tilde{x}_j^k+ \tilde{\lambda}_{n+1} \tilde{x}_{n+1}^k = \frac{1}{2}\sum_{j=1}^{2n} \tilde{w}_j L(\tilde{z}_j) 
\end{align*}
by the definition of the coefficients $\tilde{w}_j$. Then, by applying \eqref{Ip} we have
$$\mathcal{Q}_{n+1}(x^k)=\frac{1}{2} \left[2 \mathcal{I}(L(z))-{S}_{2n}(L(z)) \right].$$
So, the assertion follows by the following relations 
$$2 I(x^k)=\frac{1}{2 \pi} \int_{-\pi}^{\pi} (\cos{(\theta)})^k d \theta =\mathcal{I}(L(z)),$$
and, being $\tau=1$,
$${S}_{2n}(L(z))=2 \mathcal{G}_{n}(x^k).$$
\end{proof}
\subsection{Averaged scheme}
The anti-Szeg\"o rule developed in the previous section allows us to construct more accurate formulae. The main property \eqref{Ip} suggests us to introduce the following averaged quadrature rule
\begin{equation}\label{averaged}
\hat{\mathcal{S}}_{2n}(f)=\frac{1}{2}[\tilde{S}_{n}(f)+{S}_{n}(f)],
\end{equation}
which may lead to a higher accuracy than the single formula. 

The above rule not only allows us to approximate a given integral with greater accuracy but assent also to estimate the quadrature error provided by the Szeg\"o formula according to the following
\begin{equation}\label{estimate1}
e_n(f)=\mathcal{I}(f)-S_n(f) \simeq  \hat{\mathcal{S}}_{2n}(f)-S_n(f)=\frac{1}{2}[\tilde{S}_n(f)-S_n(f)]:=R_n(f).
\end{equation}

Let us note that formula \eqref{averaged} involves $2n$ points and requires the construction of two formulae, that is the Szeg\"o and the anti-Szeg\"o scheme. In our specific case, we know the explicit expression of zeros and weights of these two rules. However, if integral \eqref{I} presents a different measure, namely is of the form 
\begin{equation*} 
\frac{1}{2 \pi} \int_{-\pi}^\pi f(e^{i\theta}) \, d \mu(\theta),
\end{equation*}
where $\mu$ is a real-valued, bounded, non-decreasing function, with infinitely many points of increase in the interval $[-\pi,\pi]$.
Then we need to compute zeros and weights of the associated Szeg\"o and anti-Szeg\"o formula. This can be achieved by the computation of the eigenvalues of two unitary upper Hessenberg matrices, one for the Szeg\"o rule and another one for the anti-Szeg\"o formula; see \cite{Jagels2007}. Then, if we use the QR algorithm given in \cite{gragg}, the construction of these two rules requires $\mathcal{O}(\frac{3}{2}
n^2)$ flops which is cheaper than those related to the single Szeg\"o formula based on the same number of points $2n$. In fact it is $\mathcal{O}(4 n^2)$.

\section{Pointwise approximation of $Hf$}\label{sec:method}
Let us consider integral \eqref{Hilbert}. For its numerical treatment,  the most convenient approach is to write 
\begin{align}\label{Hf}
(Hf)(\phi)&= \frac{1}{2 \pi} \int_{-\pi}^\pi  
\frac{f(e^{\ii \theta})-f(e^{\ii \phi})}{\tan{\left((\theta-\phi)/2 \right)}}  
d \theta  +   \frac{\ii f(e^{\ii \phi})}{2\pi} \dashint_{-\pi}^\pi \cot{\left(\frac{\theta-\phi}{2} \right)}  d \theta \\ \nonumber 
&= \frac{1}{2 \pi} \int_{-\pi}^\pi  
\frac{f(e^{\ii \theta})-f(e^{\ii \phi})}{\tan{\left((\theta-\phi)/2 \right)}}  
d \theta, 
\end{align}
being 
$$\dashint_{-\pi}^\pi \cot{\left(\frac{\theta-\phi}{2} \right)}  d \theta=0.$$
In fact, in this way, we have subtracted out the singularity as
$$\lim_{\theta \to \phi} \frac{f(e^{\ii \theta})-f(e^{\ii \phi})}{\tan{\left((\theta-\phi)/2 \right)}}=2 f'(e^{\ii \phi}),$$  so that the integrand function of \eqref{Hf}  is continuous in $\mathbb{T} \times \mathbb{T}$.

A natural application of the formulae \eqref{szego} and \eqref{anti-szego} yields
\begin{equation} \label{Hn}
(H_nf)(\phi)= \frac{1}{n} \sum_{k=1}^n   \frac{f(z_k)-f(e^{\ii \phi})}{{\tan{\left((\theta_k-\phi)/2 \right)}}}, \qquad \theta_k= \arg(z_k),
\end{equation}
and 
\begin{equation} \label{tildeHn}
(\tilde{H}_nf)(\phi)= \frac{1}{n} \sum_{k=1}^n   \frac{f(\tilde{z}_k)-f(e^{\ii \phi})}{{\tan{\left((\tilde{\theta}_k-\phi)/2 \right)}}}, \qquad \tilde{\theta}_k= \arg(\tilde{z}_k),
\end{equation} 
respectively. However, it is not assured that the denominators do not vanish, namely, the conditions $\theta_k \neq \phi$ and $\tilde{\theta}_k \neq \phi$ are satisfied for any $k$. Moreover, even if it is, $\theta_k$ or $\tilde{\theta}_k$ could be very close to $\phi$ for some $k$, so that we may have numerical instability. 
The next example shows the difficulties we may have with formulae \eqref{Hn} and \eqref{tildeHn} and, in the next subsection, we propose a suitable approximation of \eqref{Hf}, to overcome these difficulties.  
\begin{example}\label{example1}
\rm
Let us approximate the following integral 
$$
({Hf_0})(\phi)= \frac{1}{2 \pi} \dashint_{-\pi}^\pi  
\cot{\left(\frac{\theta-\phi}{2} \right)} \, f_0(\theta) \, d \theta, \quad f_0(\theta)=e^{2 \cos{\theta}}, \qquad  \phi \in \mathbb{T}.$$
The table \ref{tab:example1} reports, for increasing values of $n$,  the approximating values of the above integral furnished by \eqref{Hn} and \eqref{tildeHn}, at two points $\phi=\dfrac{\pi}{16}, \dfrac{\pi}{32}$. The numerical instability is evident. For specific values of $n$, as $n=16$ for $\phi=\pi/16$ and $n=32$ for $\phi=32$, Szeg\"o formula provides a completely wrong value, and the anti-Szeg\"o formula does not converge for $n$ sufficiently large, a convergence which instead was shown for small values of $n$. Inspecting the algorithm, it was noted that the reason for this trend is precisely linked to the presence of quadrature nodes very close to the evaluation point $\phi$. 
\begin{table}[ht]
\caption{Numerical results for Example \ref{example1}.}
\label{tab:example1}
\begin{center}
\begin{tabular}{c | c c c}
$\phi$ & $n$ & $(H_nf_0)(\phi) $ & $ (\tilde{H}_nf_0)(\phi)|$  \\
\hline 
$\pi/16$ & 4 	  & -1.465154871646335e+00 	  & -1.486566736150962e+00 \\
& 8 	  & -1.475854994149514e+00 	  & -1.475860803898647e+00 \\
& 16 	  & -1.129061160339895e+00 	  & -1.475857899024084e+00 \\
& 32 	  & -1.475857899024079e+00 	  & -1.302459529681990e+00 \\
& 64 	  & -1.475857899024079e+00 	  & -1.389158714353031e+00 \\
& 128 	  & -1.475857899024163e+00 	  & -1.432508306688584e+00 \\
& 256 	  & -1.475857899024329e+00 	  & -1.454183102856528e+00 \\
\hline
$\pi/32$ & 4 	  & -7.489317111388031e-01 	  & -7.597532661507003e-01 \\
& 8 	  & -7.543395598939024e-01 	  & -7.543424886447505e-01 \\
& 16 	  & -7.543410242693274e-01 	  & -7.543410242693283e-01 \\
& 32 	  & -6.646769412718797e-01 	  & -7.543410242693276e-01 \\
& 64 	  & -7.543410242693263e-01 	  & -7.095089827706016e-01 \\
& 128 	  & -7.543410242693689e-01 	  & -7.319250035199811e-01 \\
& 256 	  & -7.543410242694539e-01 	  & -7.431330138947536e-01 \\
\end{tabular}
\end{center}
\end{table}
\end{example}

\subsection{An algorithm based on a prescribed node}
The main idea  is to approximate the Hilbert transform \eqref{Hf} by using a Szeg\"o formula having a prescribed node $z_p=e^{\ii \theta_p}$ which is sufficiently far from the point $e^{\ii \phi}$. 
As already observed, this is possible by fixing the parameter $\tau$ as in \eqref{taup}; see \cite{Jagels2007} for further details. We have just to fix the right value of $\theta_p$. 

Taking into account that Szeg\"o and anti-Szeg\"o nodes are such that 
$$|\arg(\tilde{z}_{k+1})- \arg(\tilde{z}_{k})|= |\arg(z_{k+1})- \arg(z_{k})|=\frac{2 \pi}{n}, \quad \forall k=1,\dots,n-1,$$
and
$$|\arg(z_k)-\arg(\tilde{z}_k)|=\frac{\pi}{n}, \quad  \forall k=1,\dots,n,$$ 
a good choice could be to define
\begin{equation*}
\theta_p=\phi+\frac{\pi}{4n}.
\end{equation*}
Hence, we introduce the following approximations
\begin{equation}\label{Hnp}
(H^{\textrm{pr}}_nf)(\phi)= \frac{1}{n} \sum_{k=1}^n   \frac{f(z_k)-f(e^{\ii \phi})}{{\tan{\left((\theta_k-\phi)/2 \right)}}}, \qquad \theta_k= \arg(z_k),
\end{equation}
and
\begin{equation}\label{antiHnp}
(\tilde{H}^{\textrm{pr}}_nf)(\phi)= \frac{1}{n} \sum_{k=1}^n   \frac{f(\tilde{z}_k)-f(e^{\ii \phi})}{{\tan{\left((\tilde{\theta}_k-\phi)/2 \right)}}}, \qquad \tilde{\theta}_k= \arg(\tilde{z}_k),
\end{equation}
where the nodes are
$$z_k=e^{\ii (\arg(\tau)+2k\pi)/n}, \quad \tilde{z}_k=e^{\ii (\arg(\tau)+(2k+1)\pi)/n}, \quad  \tau=-e^{\ii n(\phi+\frac{\pi}{4n})}, \quad k=1,\dots,n.$$ 
An average of the two formulae allow us to define the following rule
\begin{equation}\label{mediaHnp}
(\hat{H}^{\textrm{pr}}_nf)(\phi)=\frac{1}{2}\left((H^{\textrm{pr}}_nf)(\phi)+(\tilde{H}^{\textrm{pr}}_nf)(\phi)\right),
\end{equation}
and estimate the quadrature rule of \eqref{Hnp} as
\begin{equation}\label{estimate2}
r_n(f):= \frac{1}{2}[\tilde{H}^{\textrm{pr}}_nf-H^{\textrm{pr}}_nf].
\end{equation}
Next theorem shows the pointwise stability of the sequences of the discrete operators $\{H^{\textrm{pr}}_{n}f\}_{ n\in \mathbb{N}}$ and $\{\tilde{H}^{\textrm{pr}}_{n}f\}_{ n\in \mathbb{N}}$, and consequently those of $\{\hat{H}^{\textrm{pr}}_nf\}_{n \in \mathbb{N}}$, except for a factor $\log n$. 
\begin{theorem}\label{teo:stabilityHn}
Let $f \in C^0_{2 \pi}(\mathbb{T})$. Then, for each $\phi \in \mathbb{T}$, formula \eqref{Hnp} and \eqref{antiHnp} are such that
\begin{equation}\label{boundHp}
|(H^{\textrm{pr}}_{n} f)(\phi)| \leq \C \log n \, \|f\|_\infty,
\end{equation}
and
\begin{equation}\label{boundantiHp}
|(\tilde{H}^{\textrm{pr}}_{n} f)(\phi)| \leq \C \log n \, \|f\|_\infty,
\end{equation}
where $\C$ is a positive constant independent of $n$ and $f$.
\end{theorem}
\begin{proof}
First, let us prove \eqref{boundHp}. 
By definition, we have 
\begin{align*}
(H^{\textrm{pr}}_{n} f)(\phi)&= -\frac{1}{n} f(e^{\ii \phi})  \sum_{k=1}^{n}   \cot{\left(\frac{\theta_k-\phi}{2} \right)} + \frac{1}{n}   \sum_{k=1}^{n}   \frac{f(z_k)}{{\tan{\left((\theta_k-\phi)/2 \right)}}},
\end{align*}
from which we can write
\begin{align}\label{|Hn|}
|(H^{\textrm{pr}}_{n} f)(\phi)|&  \leq \frac{1}{n} \left[ |f(e^{\ii \phi})|  \sum_{k=1}^{n}   \cot{\left(\frac{\theta_k-\phi}{2} \right)}  +  
 \sum_{\substack{k=1}}^{n}  \left|f(z_k)  \cot{\left(\frac{\theta_k-\phi}{2}\right)} \right|  \right] \nonumber \\ 
 & =: \frac{1}{n} \left[S_1+S_2\right].
\end{align}
Let us now estimate the first sum. Note that, being $\theta_k=\phi+\frac{\pi}{4n}+\frac{2 k \pi}{n}$, we can write
$$\frac{\theta_k-\phi}{2}= \frac{9 \pi}{8n}+(k-1)\frac{\pi}{n}.$$
Therefore, by using the well-known relation
\begin{equation*}
\frac{1}{n} \sum_{k=1}^n \cot{\left( x+\frac{(k-1) \pi}{n} \right)} =  \cot{\left(n x\right)},
\end{equation*}
we get for $x=\dfrac{9\pi}{8n}$ 
\begin{equation}\label{S1}
S_1=n \|f\|_\infty \sum_{\substack{k=1}}^{n}  \cot{\left(\frac{\theta_k-\phi}{2} \right)}=n \|f\|_\infty \cot{\left(\frac{9 \pi}{8}  \right)}=n \|f\|_\infty (\sqrt{2}+1).
\end{equation}
Let us now consider $S_2$. By using the following relation \cite[Formula 1.421]{gradshteyn2007}
$$\cot(\pi x)=\frac{1}{\pi x}+\frac{2x}{\pi} \sum_{i=1}^{\infty} \frac{1}{x^2-i^2},$$
we have
\begin{align*}
S_2&= \sum_{\substack{k=1}}^{n}  \left|f(z_k)  \cot{\left(\frac{\theta_k-\phi}{2}\right)} \right| \\ & \leq \|f\|_\infty \, \sum_{\substack{k=1}}^{n}  \left|\cot{\left(\pi \left(\frac{1+8k}{8n}\right)\right)} \right| \\ & = \frac{1}{\pi}\|f\|_\infty \, \sum_{k=1}^n \left|\frac{8n}{(1+8k)}+\frac{1+8k}{4n} \sum_{i=1}^{\infty} \frac{1}{\left(\frac{1+8k}{8n}\right)^2-i^2} \right| \\
& \leq  \C  \|f\|_\infty \, \left( n \sum_{k=1}^n \frac{1}{k}+\sum_{k=1}^n \frac{1+8k}{4n}
\sum_{i=1}^{\infty} \frac{1}{\left|\left(\frac{1+8k}{8n}\right)^2-i^2\right|} \right) \\
& \leq \mathcal{C}  \|f\|_\infty \, \left(n \log{n} + \sum_{k=1}^n \frac{1+8k}{4n} \left(
\sum_{i=1}^{\lfloor\frac{1+8k}{8n}\rfloor} \frac{1}{\left(\frac{1+8k}{8n}\right)^2-i^2}+ \sum_{i=\lfloor \frac{1+8k}{8n}\rfloor+1}^{\infty} \frac{1}{i^2-\left(\frac{1+8k}{8n}\right)^2} \right) \right) \\ & \leq \mathcal{C}  \|f\|_\infty \, n \log{n}
\end{align*}

Then, by combining the above relation with \eqref{S1} in \eqref{|Hn|}, we have the assertion. Estimate \eqref{boundantiHp} can be proved similarly.

\end{proof}

\begin{theorem}\label{teo:stimaH}
Let $f \in C^0_{2\pi}.$ Then, for any $\phi \in \mathbb{T}$, we have
\begin{equation*}
|(Hf) (\phi)| \leq \mathcal{C} \left[\|f\|_{\infty}+\int_0^\pi \frac{\omega(f, u)}{u} d u\right],
\end{equation*}
where $\mathcal{C}$ denotes a positive constant independent of $f$.
\end{theorem}
\begin{proof}
Consider the definition \eqref{Hilbert}. Using the well-known identity
\[
\frac{1}{2} \cot \left(\frac{\theta}{2}\right)=\frac{1}{\theta}+\sum_{\substack{k=-\infty \\ k \neq 0}}^{\infty}\left(\frac{1}{\theta+2 k \pi}+\frac{1}{2 k \pi}\right),
\]
we can rewrite $Hf$ as
\[
\begin{aligned}
(Hf)(\phi) & =\frac{1}{2 \pi} \dashint_{-\pi}^{\pi}  \cot \left(\frac{\theta-\phi}{2}\right) \, f\left(e^{i \theta}\right) \, d \theta \\
& =\frac{1}{\pi} \dashint_{-\pi}^{\pi} \frac{f\left(e^{i \theta}\right)}{\theta-\phi} d \theta+\frac{1}{\pi}  \dashint_{-\pi}^{\pi} f\left(e^{i \theta}\right) \sum_{\substack{k=-\infty \\ k \neq 0}}^{\infty} \left(\frac{1}{\theta-\phi+2 k \pi}+\frac{1}{2 k \pi}\right)   d \theta.
\end{aligned}
\]
The second integrand is bounded and continuous, and its integral is bounded in absolute value by $M \| f \|_\infty$ according to Holder's inequality, where $M$ is a constant independent of $f$. Hence, we focus on

\[ 
(Jf)(\phi)=\frac{1}{\pi}\dashint_{-\pi}^{\pi} \frac{f(e^ {\ii \theta})}{\theta-\phi} d \theta.
\]

Let $a>0$ such that $-\pi<-a<\phi<a<\pi$, fix $\epsilon=\frac{\pi-a}{2}$ and let us write 
\begin{align*}
\frac{1}{\pi}\dashint_{-\pi}^{\pi} \frac{f(e^ {\ii \theta})}{\theta-\phi} d \theta&= \frac{1}{\pi} \left[ \dashint_{-\pi}^{\phi-\epsilon}  \frac{f(e^ {\ii \theta})}{\theta-\phi} d \theta +\dashint_{\phi-\epsilon}^{\phi+\epsilon}  \frac{f(e^ {\ii \theta})}{\theta-\phi} d \theta +\dashint_{\phi+\epsilon}^{\pi} \frac{f(e^ {\ii \theta})}{\theta-\phi} d \theta \right] \\ & = I_1+I_2+I_3.
\end{align*}
Concerning the first integral, we have
\begin{align}\label{I1bis}
|I_1| \leq \|f\|_\infty \int_{-\pi}^{\phi-\epsilon} \frac{d \theta}{\phi-\theta} = \|f\|_\infty \log{\left(\frac{\phi+\pi}{\epsilon}\right)} \leq \|f\|_\infty  \log{\left(\frac{2(a+\pi)}{\pi-a}\right)},
\end{align}
and a similar estimate holds true for the third integral
\begin{align}\label{I3bis}
|I_3| \leq \|f\|_\infty \int_{\phi+\epsilon}^\pi \frac{d \theta}{\phi-\theta} = \|f\|_\infty \log{\left(\frac{\pi-\phi}{\epsilon}\right)} \leq \|f\|_\infty  \log{\left(\frac{2(a+\pi)}{\pi-a}\right)}.
\end{align}
Let us now consider $I_2$ and the following identity
\begin{align}\label{I2_2}
I_2= \dashint_{\phi-\epsilon}^{\phi}  \frac{f(e^ {\ii \theta})}{\theta-\phi} d \theta + \dashint_{\phi}^{\phi+\epsilon}  \frac{f(e^ {\ii \theta})}{\theta-\phi} d \theta = \int_0^\epsilon \frac{f(e^{\ii(\phi+u)})-f(e^{\ii(\phi-u)})}{u} du.
\end{align}
Introduce a function $g \in W^1$ and write
\begin{align*}
|f(e^{\ii(\phi+u)})-f(e^{\ii(\phi-u)})|&=|(f-g)(e^{\ii(\phi+u)})-(f-g)(e^{\ii(\phi-u)})+g(e^{\ii(\phi+u)})-g(e^{\ii(\phi-u)})| \\ &
\leq \|f-g\|_\infty+ \int_{\phi-u}^{\phi+u} |g'(e^{\ii y})| dy \\ & \leq \|f-g\|_\infty + u \|g'\|_\infty. 
\end{align*} 
Then, taking the infimum on $g \in W^1$ and by applying \eqref{estimateK}, we have
\begin{equation*} 
|f(e^{\ii(\phi+u)})-f(e^{\ii(\phi-u)})|  \leq \mathcal{C} K(f,u) \leq  \mathcal{C} \omega(f,u),
\end{equation*} 
from which by \eqref{I2_2} we get
\begin{equation}\label{I2bis}
|I_2| \leq \mathcal{C}\int_0^\pi \frac{\omega(f,u) }{u} du.
\end{equation}
Therefore, by combining \eqref{I1bis}, \eqref{I2bis}, and \eqref{I3bis}, we get the assertion.

%
%
%
%
\end{proof}

\begin{theorem}\label{error}
Let \( \phi \in \mathbb{T} \) be fixed and $f\in C^0_{2 \pi}$. Then, the quadrature error of the Szeg\"o rule defined in \eqref{Hnp} satisfies the following estimate 
\[
\left|e^{\textrm{pr}}_n(f, \phi)\right| \leq C\left[\log n \, E_n(f)+\int_0^{\pi/n} \frac{\omega(f, u)}{u} d u\right],
\]
where $\mathcal{C}$ is a positive constant independent of $f$ and $n$.
\end{theorem}

\begin{proof}
Let \( Q \in \Lambda_{-n-1,n-1} \) be the Laurent polynomial of best approximation for $f$. Then, by \eqref{exactness} we have 
\[
\begin{aligned}
|e_n(f, \phi)|  =|e_n(f-Q, \phi)| 
& \leq\left|\int_{-\pi}^\pi (f-Q) (e^{i \theta})  \cot \left(\frac{\theta-\phi}{2}\right) d \theta\right|+\left|H^{\textrm{pr}}_n(f-Q, \phi)\right|,
\end{aligned}
\]
and, by Theorem \ref{teo:stabilityHn} and Theorem \ref{teo:stimaH}, we obtain  
\[
\left|e_n(f, \phi)\right| \leq \mathcal{C}\left[\|f-Q\|_{\infty}+\int_0^\pi \frac{\omega(f-Q, u)}{u} d u + \log n \, E_n(f)\right].
\]
Now, by the linearity of the modulus of smoothness, we have

\begin{align}\label{omegafQ}
\int_0^\pi \frac{\omega(f-Q, u)}{u} d u & \leq  \int_0^{\pi/n} \frac{\omega(f-Q, u)}{u} d u + \|f-Q\|_{\infty} \int_{\pi/n}^\pi \frac{d u}{u} \nonumber \\
& \leq \mathcal{C} \left[  \int_0^{\pi/n} \frac{\omega(f, u)}{u} d u+\int_0^{\pi/n} \frac{\omega(Q,u)}{u} d u + \log n \|f-Q\|_{\infty} \right] \nonumber \\
& \leq \mathcal{C} \left[  \int_0^{\pi/n} \frac{\omega(f, u)}{u} d u+\frac{\pi}{n} \|Q'\|_\infty + \log n \|f-Q\|_{\infty} \right],
\end{align}
where the $\omega(Q,u)$ has been estimated by using the first inequality of \eqref{estimateK} and the definition of the $K$-functional.
Then, by using the second inequality of \eqref{estimateK}, we can also write
$$\frac{\pi}{n} \|Q'\|_\infty \leq \mathcal{C} \omega\left(f,\frac{1}{n}\right),$$
and being
$$1= \frac{1}{n} \int_{1/2n}^{1/n} \frac{dt}{t^2} \leq 2 \int_{1/2n}^{1/n} \frac{dt}{t},$$
we can deduce
$$ \frac{\pi}{n} \|Q'\|_\infty \leq \mathcal{C} \omega\left(f,\frac{1}{n}\right) \leq 2 \omega\left(f,\frac{1}{2n}\right) \leq 4 \int_{1/2n}^{1/n} \frac{\omega(f,t)}{t} dt.$$

By replacing the last estimate in \eqref{omegafQ}  we get the assertion.
\end{proof}

\begin{remark}
Note that the same error estimate can be deduced for the anti-Szeg\"o rule \eqref{antiHnp}, by virtue of Theorem \ref{teo:stabilityHn} and Theorem \ref{teo:stimaH}. A further consequence of Theorem \ref{error} is that if $\phi \in \mathbb{T}$ is fixed, for any function $f$ such that  
$$\int_0^{\pi/n} \frac{\omega(f,u)}{u} du < \infty,$$ 
we have convergence. Moreover, if $f \in W^r$, then the error behaves as $\mathcal{O}\left(\frac{\log n}{n^r} \right).$ 
We remark that this is the optimal order of convergence we can obtain. In fact, according to \cite{stolle1992}, there exist a positive constant $\mathcal{C}$ (depending only on the parameter $r$) such that for every $n$-point quadrature formula $Q_n$ based on the subtraction of the singularity and approximating a Cauchy singular integral,   the following estimate holds true 
$$\sup_{\phi \in \mathbb{T}}  \varrho_r(e_n(\phi))  \geq \mathcal{C} \frac{\log n}{n^r},$$
where $\varrho_r(e_n(\phi)) = \sup  \{|e_n(f, \phi)| : f \in W^r\} $ is the Peano constant of order $r$ related to the quadrature $Q_n$.
\end{remark}

\section{Numerical tests}\label{sec:tests}
This section is designed to present several numerical experiments that   illustrate the performance of our algorithm. 
In each example, we first write the integral \eqref{Cauchy} as 
$$(If)(z)=(Hf)(z)+\ii \, \mathcal{I}(f),$$
where $H$ and $\mathcal{I}$ are given in \eqref{Hilbert} and \eqref{I}, respectively. Then, for the circular Hilbert transform we
evaluate the sums \eqref{Hnp}, \eqref{antiHnp}, and \eqref{mediaHnp}  in $10^2$ equidistant points in $\phi_i \in [-\pi,\pi]$ and compute the discrete errors
\begin{align}\label{errors}
\epsilon_n(f) &= \|Hf-H^{\textrm{pr}}_nf\|, \qquad 
\tilde{\epsilon}_n(f) = \|Hf-\tilde{H}^{\textrm{pr}}_nf\|, \qquad 
\hat{\epsilon}_n(f) =  \|Hf -\hat{H}^{\textrm{pr}}_{n}f\|,
\end{align}
where here $\|\cdot\|$ denotes the discrete infinity norm and  $Hf$ is the exact value of the integral. For the second integral, we approximate it with the Szeg\"o and anti-Szeg\"o rule defined in \eqref{szego} and \eqref{anti-szego}, and evaluate the errors
\begin{align}\label{errorsI}
e_n(f) &= \mathcal{I}f-S_nf, \qquad 
\tilde{e}_n(f) = \mathcal{I}f-\tilde{S}_nf, \qquad 
\hat{e}_n(f) =  \mathcal{I}f -\hat{S}_{n}f.
\end{align}

When the exact values of the integrals are not known,  we consider the exact value the one obtained with a large value of $n$ that we fix case by case. We point out that this choice does not affect the results since  our quadrature schemes are stable and convergent.

All the computed examples were carried out in  Matlab R2023b in double precision  on an Intel Core i7-2600 system (8 cores), under the Debian GNU/Linux operating system.

\begin{example}\label{test0}
\rm
Let us start by considering again the integral of Example \ref{example1}, that is the circular Hilbert transform of the function $f_0(\phi)=e^{2 \cos{\phi}}$. 
To make a comparison with Table \ref{tab:example1}, we have computed the values of the integral in $\phi=\pi/16, \pi/32$. Table \ref{tab:test0}  contains such values. It is deduced that the algorithm with the prescribed node avoids instability. Moreover, by Table \ref{tab:test0bis} which reports  the errors \eqref{errors} considering as exact the approximated value with $n=128$, one can infer that the averaged rule furnishes very accurate results. 
\begin{table}[ht]
\caption{Numerical results for Example \ref{test0}.}
\label{tab:test0}
\begin{center}
\begin{tabular}{c | c c c| c c c}
$n$ & $\phi$ & $(H_nf_0)(\phi) $ & $ (\tilde{H}_nf_0)(\phi)$   \\
\hline 
  4 & $\frac{\pi}{16}$	  & -1.622605841221501e+00 	  & -1.329104147077534e+00  \\
 8 &	  & -1.475904319788829e+00 	  & -1.475811478259103e+00  \\
 16 &	  & -1.475857899023998e+00 	  & -1.475857899024163e+00 \\
 32 &	  & -1.475857899024072e+00 	  & -1.475857899024079e+00 \\
 64 &	  & -1.475857899024066e+00 	  & -1.475857899024052e+00 \\
 128 &	  & -1.475857899024082e+00 	  & -1.475857899024035e+00 \\
 256 	&  & -1.475857899024136e+00 	  & -1.475857899024142e+00 \\
\hline
  4 &$\frac{\pi}{32}$ 	  & -8.930293238806029e-01 	  & -6.157479708830708e-01 \\
 8 	&  & -7.544098378965085e-01 	  & -7.542722106421451e-01 \\
 16 	&  & -7.543410242694044e-01 	  & -7.543410242692503e-01 \\
 32 &	  & -7.543410242693269e-01 	  & -7.543410242693250e-01 \\
 64 	&  & -7.543410242693245e-01 	  & -7.543410242693196e-01 \\
 128 &	  & -7.543410242692896e-01 	  & -7.543410242693149e-01 \\
 256 &	  & -7.543410242693608e-01 	  & -7.543410242692843e-01 \\
\end{tabular}
\end{center}
\end{table}

\begin{table}[ht]
\caption{Numerical results for Example \ref{test0}.}
\label{tab:test0bis}
\begin{center}
\begin{tabular}{c | c c c}
$n$  & $\epsilon_n(f) $ & $\tilde{\epsilon}_n(f) $ & $\hat{\epsilon}(f)$ \\
\hline 
  4  & 1.47e-01 	  & 1.47e-01 	  & 6.66e-05 \\
 8  & 6.88e-05 	  & 6.88e-05 	  & 2.02e-13 	\\
 16 & 1.77e-13 	  & 1.41e-13 	  & 9.57e-14\\
\end{tabular}
\end{center}
\end{table}
\end{example}

\begin{example}\label{test1}
\rm
Let us evaluate the following integral
\begin{equation*}
I(z)=\frac{1}{\pi} \dashint_{\Gamma} \frac{\ln(1+\cos{t}+\sin^2(t/2))}{t-z} \, dt, \qquad z \in \Gamma.
\end{equation*}
Setting $f_1(e^{\ii \theta})= \ln{\left(\frac{3}{2}+\frac{e^{i \theta}}{4}+\frac{e^{-i\theta}}{4}\right)}$ and according to \eqref{If2} it can be rewritten as
\begin{equation*}
(If_1)(e^{\ii \phi})=\frac{1}{2 \pi} \dashint_{-\pi}^\pi  
\cot{\left(\frac{\theta-\phi}{2} \right)} f_1(e^{\ii \theta})  d \theta + \frac{\ii}{2\pi} \int_{-\pi}^\pi  f_1(e^{\ii \theta}) d \theta:=(Hf_1)(e^{\ii \phi})+ \ii \, \mathcal{I}(f_1).
\end{equation*}
Table \ref{tab:test1} reports the absolute errors \eqref{errors} and \eqref{errorsI} we get for increasing value of $n$ and by considering as exact value the one obtained with $n=64$. In Figure \ref{fig:sol_errors}, we display the errors provided by the Szego and anti-Szego formula for the Hilbert transform when $n=4$. It is evident that the errors have opposite sign. This fact allows us to gain accuracy with the average rule almost reaching the machine precision with $n=8$ .
\begin{table}[ht!]
\caption{Numerical results for Example \ref{test1}.}
\label{tab:test1}
\begin{center}
\begin{tabular}{c | c c c | c c c }
$n$ & $\epsilon_n(f_1) $ & $\tilde{\epsilon}_n(f_1) $ & $\hat{\epsilon}_n(f_1)$ & $e_n(f_1) $ & $\tilde{e}_n(f_1) $ & $\hat{e}_n(f_1)$\\
\hline
4 	  & 5.69e-04 	  & 5.69e-04 	  & 2.55e-07 &  4.33e-04 	& -4.33e-04&  	 -1.88e-07 	\\
8 	  & 2.47e-07 	  & 2.47e-07 	  & 9.84e-14 & 1.88e-07  & 	 -1.88e-07 &	 -7.07e-14 	\\
16 	  & 9.52e-14 	  & 9.53e-14 	  & 4.91e-15 & 7.13e-14 	& -7.01e-14 	& 6.11e-16 	 \\
\end{tabular}
\end{center}
\end{table}
\begin{figure}[ht!]
\centering
\includegraphics[width=0.48\textwidth]{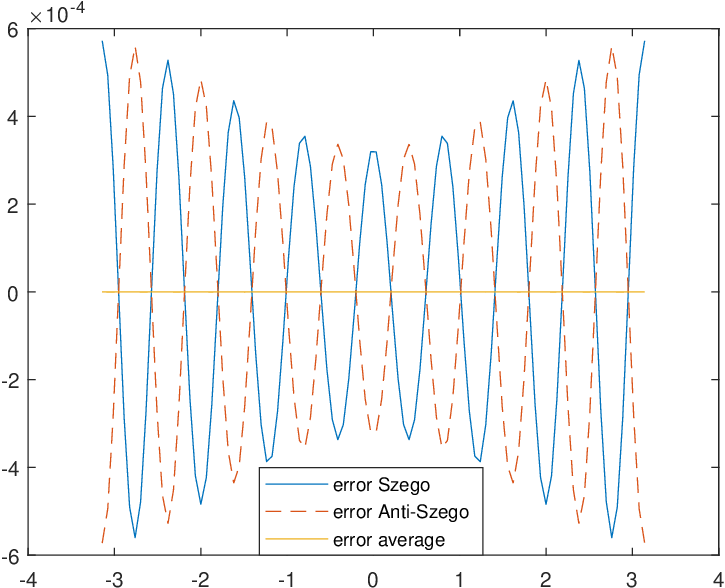}
\caption{Graph of the errors for Example~\ref{test1} with $n=4$.}
\label{fig:sol_errors}
\end{figure}
\end{example}
\begin{example}\label{test2}
\rm
Let us approximate the following integral
\begin{equation*}
\frac{1}{2\pi} \dashint_{\Gamma} \frac{\ln(5+4\cos{t})}{t-z} \, dt, \qquad z \in \Gamma,
\end{equation*}
or equivalently 
\begin{equation*}
(If_2)(e^{\ii \phi})=\frac{1}{2 \pi} \int_{-\pi}^\pi  
\cot{\left(\frac{\theta-\phi}{2} \right)} f_2(e^{\ii \theta})  d \theta  + \frac{\ii}{2\pi} \int_{-\pi}^\pi  f_2(e^{\ii \theta}) d \theta,
\end{equation*}
with $f_2(e^{\ii \theta})=\ln{\left(5+2{e^{i \theta}+2e^{-i\theta}}\right)} $. Table \ref{tab:test3} reports not only the discrete errors but also an estimate of the error provided by the approximation given by the Szeg\"o rule as in \eqref{estimate1} and \eqref{estimate2}. By comparing the fifth column with the second one and the last column with the sixth one, we can deduce that the couple Szeg\"o and anti-Szeg\"o rule provide a practical good way the get such estimate.  
\begin{table}[ht!]
\caption{Numerical results for Example \ref{test2}.}
\label{tab:test3}
\begin{center}
\begin{tabular}{c | c c c c | c c c c}
$n$ & $\epsilon_n(f_2) $ & $\tilde{\epsilon}_n(f_2) $ & $\hat{\epsilon}_n(f_2)$ & $r_n(f_2)$ & $e_n(f_2)$ & $\tilde{e}_n(f_2) $ & $\hat{e}_n(f_2)$ & $R_n(f_2)$\\ 
\hline
4 	  & 3.69e-02 	  & 3.60e-02 	  & 1.28e-03 	 & 3.64e-02  & 1.52e-02 	& -1.61e-02 	 &-4.89e-04 	 &-1.56e-02 \\
8 	  & 1.25e-03 	  & 1.25e-03 	  & 2.66e-06 	 & 1.25e-03 & 4.87e-04 	& -4.89e-04 	 &-9.54e-07 	 &-4.88e-04 \\
16 	  & 2.58e-06 	  & 2.58e-06 	  & 2.10e-11 	 & 2.58e-06 & 9.54e-07 	& -9.54e-07 	 &-7.28e-12 	 &-9.54e-07\\
32 	  & 2.03e-11 	  & 2.03e-11 	  & 3.45e-14 	 & 2.03e-11 & 7.28e-12 	& -7.28e-12 	 &2.22e-16 	 &-7.28e-12\\
\end{tabular}
\end{center}
\end{table}
\end{example}
\begin{example}\label{test3}
\rm
Let us consider now an example in which the known periodic function is not smooth
\begin{align*}
(If_3)(z)&=\frac{1}{2\pi} \dashint_{\Gamma} \frac{|1+\cos{t}|^{\frac{5}{2}}}{t-z} \, dt = \frac{1}{2 \pi} \int_{-\pi}^\pi  
\cot{\left(\frac{\theta-\phi}{2} \right)} f_3(e^{\ii \theta})  d \theta  + \frac{\ii}{2\pi} \int_{-\pi}^\pi  f_3(e^{\ii \theta}) d \theta,
\end{align*}
where $f_3(e^{\ii \theta})=|1+\frac{e^{\ii \theta}}{2}+\frac{1}{2e^{\ii \theta}}|^{5/2}$.
As we can see by Table \ref{tab:test4} the convergence is not so fast as in the previous example. This is due to the precence of a function of class $W^2$. However, also in this case we can see that the averaged rule improves the accuracy.   

\begin{table}[ht!]
\caption{Numerical results for Example \ref{test3}.}
\label{tab:test4}
\begin{center}
\begin{tabular}{c | c c c c | c c c c}
$n$ & $\epsilon_n(f_3) $ & $\tilde{\epsilon}_n(f_3) $ & $\hat{\epsilon}_n(f_3)$ & $r_n(f_3)$ & $e_n(f_3)$ & $\tilde{e}_n(f_3) $ & $\hat{e}_n(f_3)$ & $R_n(f_3)$\\ 
\hline
4 	  & 9.97e-03 	  & 1.00e-02 	  & 1.89e-04 	 & 9.97e-03 &  6.34e-03 	& -6.46e-03 	 &-6.03e-05 	 &-6.40e-03 \\
8 	  & 1.86e-04 	  & 1.83e-04 	  & 5.31e-06 	 & 1.85e-04  & 5.86e-05 	& -6.03e-05 	 &-8.47e-07 	 &-5.95e-05 \\
16 	  & 4.64e-06 	  & 4.67e-06 	  & 1.62e-07 	 & 4.64e-06 & 8.21e-07 	& -8.47e-07 	 &-1.29e-08 	 &-8.34e-07 \\
32 	  & 1.41e-07 	  & 1.41e-07 	  & 5.02e-09 	 & 1.41e-07 & 1.25e-08 	& -1.29e-08 	 &-2.02e-10 	 &-1.27e-08 \\
64 	  & 3.73e-09 	  & 3.42e-09 	  & 1.57e-10 	 & 3.58e-09 & 1.92e-10 	& -2.02e-10 	 &-4.74e-12 	 &-1.97e-10 \\
128 	  & 1.16e-10 	  & 1.07e-10 	  & 4.79e-12 	 & 1.12e-10 & 1.55e-12 	& -4.67e-12 	 &-1.56e-12 	 &-3.11e-12 \\
256 	  & 3.53e-12 	  & 3.51e-12 	  & 8.11e-13 	 & 3.49e-12 & -1.20e-12 	& -1.36e-12 	 &-1.28e-12 	 &-8.13e-14 
\end{tabular}
\end{center}
\end{table}
\end{example}
\begin{example}\label{test4}
\rm
Let us test our quadrature schemes  on the following integral
\begin{equation*}
(If_4)(z)=\frac{1}{2\pi} \dashint_{\Gamma} \frac{|\sin{t}|^{\frac{7}{2}}}{t-z} \, dt, \qquad z \in \Gamma,
\end{equation*}
where $f_4(e^{\ii \theta})= |\frac{e^{\ii \theta}-e^{-\ii \theta}}{2\ii} |^{7/2} \in W^3$.
Table \ref{tab:test5} contains our numerical results which are better than the theoretical expectation.  

\begin{table}[ht!]
\caption{Numerical results for Example \ref{test4}.}
\label{tab:test5}
\begin{center}
\begin{tabular}{c | c c c| c c c}
$n$ & $\epsilon_n(f_4) $ & $\tilde{\epsilon}_n(f_4) $ & $\hat{\epsilon}_n(f_4)$  & $e_n(f_4)$ & $\tilde{e}_n(f_4) $ & $\hat{e}_n(f_4)$  \\ 
\hline
8 	  & 2.55e-03 	  & 2.48e-03 	  & 1.90e-04 & -1.12e-03 	& 1.21e-03 	 &4.45e-05 \\
16 	  & 1.86e-04 	  & 1.81e-04 	  & 1.61e-05 &  -4.08e-05 	& 4.45e-05 	 &1.89e-06 \\
32 	  & 1.32e-05 	  & 1.34e-05 	  & 1.40e-06 & -1.72e-06 	& 1.89e-06 	 &8.26e-08 	\\
64 	  & 1.15e-06 	  & 1.08e-06 	  & 1.23e-07 & -7.53e-08 	& 8.26e-08 	 &3.65e-09 \\
128 	  & 1.01e-07 	  & 8.06e-08 	  & 1.01e-08 & -3.31e-09 	& 3.65e-09 	 &1.66e-10 \\
256 	  & 8.19e-09 	  & 7.85e-09 	  & 1.73e-10 & -1.41e-10 	& 1.67e-10 	 &1.29e-11 
\end{tabular}
\end{center}
\end{table}
\end{example}

\section{Conclusions}\label{sec:conclusions}
In this paper, we have proposed quadrature rules of Szeg\"o and anti-Szeg\"o type for the approximation of the Hilbert transform defined on the unit circle. The schemes are  suitable constructed to avoid that the quadrature nodes coincide or are very close to the singularity. An averaged rule is also proposed allowing for better accuracy and a reduction of the computational cost with respect to the native formulae.

As future perspective of research, we think that we can extend the same procedure to the case when other measures appear in the Hilbert transform. We also believe that these quadrature rules can be applied to the numerical solution of Cauchy integral equations define on the unit circle and currently unexplored. Finally, a specific spectral analysis of the matrices and matrix-sequences considered in this work is a topic for future investigation, together with fast accurate eigenvalue solvers in the spirit of \cite{bogoya2024fast}.

\section*{Acknowledgments}
The authors are members of the Gruppo Nazionale Calcolo Scientifico-Istituto Nazionale di Alta Matematica (GNCS-INdAM) and are partially supported by the INdAM-GNCS 2024 project ``Algebra lineare numerica per problemi di grandi dimensioni: aspetti teorici e applicazioni''.
Luisa Fermo is also a member of the TAA-UMI Research Group and is partially supported by the PRIN 2022 PNRR project no. P20229RMLB financed by the
European Union - NextGeneration EU and by the Italian Ministry of University and Research (MUR).
This research has been accomplished within “Research ITalian network on Approximation” (RITA).
\bibliographystyle{plain}      
\bibliography{biblio}

\begin{thebibliography}{10}

\bibitem{Benjamin1967}
T.~B. Benjamin.
\newblock Internal waves of permanent form in fluids of great depth.
\newblock {\em J. Fluid Mech.}, 29(3):559--592, 1967.

\bibitem{bogoya2024}
M.~Bogoya, J.~Gasca, and S.~Grudsky.
\newblock Eigenvalue asymptotic expansion for non-{H}ermitian tetradiagonal
  {T}oeplitz matrices with real spectrum.
\newblock {\em J. Math. Anal. Appl.}, 531(1):127816, 2024.

\bibitem{bogoya2024fast}
M.~Bogoya, S.~M. Grudsky, and S.~Serra-Capizzano.
\newblock Fast non-{H}ermitian {T}oeplitz eigenvalue computations, joining
  matrixless algorithms and {FDE} approximation matrices.
\newblock {\em SIAM J. Matrix Anal. Appl.}, 45(1):284--305, 2024.

\bibitem{Bulthell2001}
A.~Bultheel, L.~Daruis, and P.~Gonzalez-Vera.
\newblock A connection between quadrature formulas on the unit circle and the
  interval $[-1,1]$.
\newblock {\em J. Comput. Appl. Math.}, 132(1):1--14, 2001.

\bibitem{Bulthell1991}
A.~Bultheel, P.~Gonz{\'a}lez-Vera, E.~Hendriksen, and O.~Nj{\aa}stad.
\newblock Orthogonality and quadrature on the unit circle.
\newblock {\em IMACS annals on Computing and {A}pplied {M}athematics},
  9:205--210, 1991.

\bibitem{chawla1974}
M.M. Chawla and T.R. Ramakrishnan.
\newblock Numerical evaluation of integrals of periodic functions with {C}auchy
  and {P}oisson type kernels.
\newblock {\em Numer. Math.}, 22:317--323, 1974.

\bibitem{coussement2008}
E.~Coussement, J.~Coussement, and W.~Van~Assche.
\newblock Asymptotic zero distribution for a class of multiple orthogonal
  polynomials.
\newblock {\em Trans. Amer. Math. Soc.}, 360(10):5571--5588, 2008.

\bibitem{DFR2020}
P.~Díaz~de Alba, L.~Fermo, and G.~Rodriguez.
\newblock Solution of second kind {F}redholm integral equations by means of
  {G}auss and anti-{G}auss quadrature rules.
\newblock {\em Numer. Math.}, 146(4):699--728, 2020.

\bibitem{Gauss}
C.~F. Gauss.
\newblock Methodus nova integralium valores per approximationem inveniendi.
\newblock {\em Comm. Soc. R. Sci. G\"ottingen Recens.}, 3:39--76, 1814.
\newblock Werke \textbf{3}, 163--196, (1866).

\bibitem{Gautschi}
W.~Gautschi.
\newblock A survey of {G}auss-{C}hristoffel quadrature formulae.
\newblock In P.~L. Butzer and F.~Feh\'er, editors, {\em E. B. Christoffel. The
  Influence of His Work on Mathematics and the Physical Sciences}, pages
  72--147. Springer, 1981.

\bibitem{Golinskii2007}
L.~Golinskii and S.~Serra-Capizzano.
\newblock The asymptotic properties of the spectrum of nonsymmetrically
  perturbed {J}acobi matrix sequences.
\newblock {\em J. Approx. Theory}, 144(1):84--102, 2007.

\bibitem{gradshteyn2007}
I.~S. Gradshteyn and I.~M. Ryzhik.
\newblock {\em Table of integrals, series, and products}.
\newblock Elsevier/Academic Press, Amsterdam, seventh edition, 2007.

\bibitem{gragg}
W.~B. Gragg.
\newblock Positive definite {T}oeplitz matrices, the {A}rnoldi process for
  isometric operators, and {G}aussian quadrature on the unit circle.
\newblock {\em J. Comput. Appl. Math.}, 46:183--198, 1993.

\bibitem{Jagels2007}
C.~Jagels and L.~Reichel.
\newblock Szeg\"o-{L}obatto quadrature rules.
\newblock {\em J. Comput. Appl. Math.}, 200(1):116--126, 2007.

\bibitem{jin1988}
D.~Jin-Yuan.
\newblock Quadrature formulas for singular integrals with {H}ilbert kernel.
\newblock {\em J. Comput. Math.}, pages 205--225, 1988.

\bibitem{jinyuan1998}
D.~Jinyuan.
\newblock Quadrature formulas of quasi-interpolation type for singular
  integrals with {H}ilbert kernel.
\newblock {\em J. Approx. Theory}, 93(2):231--257, 1998.

\bibitem{jones}
W.B. Jones, O.~Nj{\aa}stad, and W.J. Thron.
\newblock Moment {T}heory, {O}rthogonal {P}olynomials, {Q}uadrature, and
  {C}ontinued {F}ractions {A}ssociated with the unit {C}ircle.
\newblock {\em Bull. Lond. Math. Soc.}, 21:113--152, 1989.

\bibitem{KimReichel}
S.-M. Kim and L.~Reichel.
\newblock Anti-{S}zeg\"o quadrature rules.
\newblock {\em Math. Comp.}, 76(258):795--810, 2007.

\bibitem{Kress99}
R.~Kress.
\newblock {\em Linear Integral Equation}.
\newblock Springer, 1999.

\bibitem{Kuijla2001}
A.~B.~J. Kuijlaars and S.~{Serra-Capizzano}.
\newblock Asymptotic zero distribution of orthogonal polynomials with
  discontinuously varying recurrence coefficients.
\newblock {\em J. Approx. Theory}, 113(1):142--155, 2001.

\bibitem{Laurie1}
D.~P. Laurie.
\newblock Anti-{G}aussian quadrature formulas.
\newblock {\em Math. Comp.}, 65:739--747, 1996.

\bibitem{Marple1999}
L.~Marple.
\newblock Computing the discrete-time analytic signal via {FFT}.
\newblock {\em IEEE Trans. Signal Process.}, 47(9):2600--2603, 1999.

\bibitem{micchelli2013}
C.~A Micchelli, Y.~Xu, and B.~Yu.
\newblock On computing with the {H}ilbert spline transform.
\newblock {\em Adv. Comput. Math.}, 38:623--646, 2013.

\bibitem{Notaris2022}
S.~E. Notaris.
\newblock Anti-{G}aussian quadrature formulae of {C}hebyshev type.
\newblock {\em Math. Comp.}, 91(338):2803--2816, 2022.

\bibitem{Olver2011}
S.~Olver.
\newblock Computing the {H}ilbert transform and its inverse.
\newblock {\em Math. Comp.}, 80(275):1745--1767, 2011.

\bibitem{Ono1975}
H.~Ono.
\newblock Algebraic solitary waves in stratified fluids.
\newblock {\em J. Phys. Soc. Jap.}, 39:1082--1091, 1975.

\bibitem{PranicReichel}
M.~S. Prani\'c and L.~Reichel.
\newblock Generalized anti-{G}auss quadrature rules.
\newblock {\em J. Comput. Appl. Math.}, 284:235--243, 2015.

\bibitem{ReichelSpalevic2021}
L.~Reichel and M.M. Spalevi\'c.
\newblock A new representation of generalized ave\-ra\-ged {G}auss quadrature
  rules.
\newblock {\em Appl. Numer. Math.}, 165:614--619, 2021.

\bibitem{Simon2007}
B.~Simon.
\newblock Rank one perturbations and the zeros of paraorthogonal polynomials on
  the unit circle.
\newblock {\em J. Math. Anal. Appl.}, 329(1):376--382, 2007.

\bibitem{stolle1992}
H.W. Stolle and R.~Strauss.
\newblock On the numerical integration of certain singular integrals.
\newblock {\em Computing}, 48(2):177--189, 1992.

\bibitem{Sun2019}
X.~Sun and P.~Dang.
\newblock Numerical stability of circular {H}ilbert transform and its
  application to signal decomposition.
\newblock {\em Appl. Math. Comput.}, 359:357--373, 2019.

\bibitem{Tyrty2003}
E.E. Tyrtyshnikov and N.L. Zamarashkin.
\newblock A general equidistribution theorem for the roots of orthogonal
  polynomials.
\newblock {\em Linear Algebra Appl.}, 366:433--439, 2003.

\end{thebibliography}
\end{document}